\documentclass[reqno,12pt]{article}

\title{The snake in the Brownian sphere}

\author{Omer Angel\footnote{University of British Columbia, Department of Mathematics} 
\and 
Emmanuel Jacob\footnote{\'Ecole Normale Sup\'erieure de Lyon, Unit\'e de Math\'ematiques Pures et Appliqu\'ees} 
\and 
Brett Kolesnik\footnote{University of Warwick, Department of Statistics} 
\and 
Gr\'egory Miermont\footnote{\'Ecole Normale Sup\'erieure de Lyon, Unit\'e de Math\'ematiques Pures et Appliqu\'ees, and Institut Universitaire de France}}

\date{February 2025}

\usepackage[T1]{fontenc}
\pdfoutput=1
\usepackage{amsthm,amsmath,amssymb}
\usepackage{microtype,graphicx,dsfont}
\usepackage[hidelinks]{hyperref}
\usepackage{enumitem}
\usepackage[nameinlink,capitalise]{cleveref}
\usepackage[font=footnotesize,margin=2cm]{caption}
\usepackage{cite}
\usepackage{tikz}
\usepackage{tikz-cd}

\crefname{theorem}{Theorem}{Theorems}
\crefname{thm}{Theorem}{Theorems}
\crefname{mainthm}{Theorem}{Theorems}
\crefname{lemma}{Lemma}{Lemmas}
\crefname{lem}{Lemma}{Lemmas}
\crefname{remark}{Remark}{Remarks}
\crefname{claim}{Claim}{Claims}
\crefname{subclaim}{Sub-claim}{Sub-claims}
\crefname{prop}{Proposition}{Propositions}
\crefname{proposition}{Proposition}{Propositions}
\crefname{defn}{Definition}{Definitions}
\crefname{corollary}{Corollary}{Corollaries}
\crefname{conjecture}{Conjecture}{Conjectures}
\crefname{question}{Question}{Questions}
\crefname{chapter}{Chapter}{Chapters}
\crefname{section}{Section}{Sections}
\crefname{figure}{Figure}{Figures}
\newcommand{\crefrangeconjunction}{--}

\theoremstyle{plain}

\newtheorem{thm}{Theorem}
\newtheorem*{thm*}{Theorem}
\newtheorem{lemma}[thm]{Lemma}

\newtheorem{corollary}[thm]{Corollary}
\newtheorem{prop}[thm]{Proposition}

\theoremstyle{definition}

\theoremstyle{remark}
\newtheorem*{remark}{Remark}

\newcommand{\eps}{\varepsilon}
\renewcommand{\P}{{\mathbb P}}

\newcommand{\R}{{\mathbb R}}

\newcommand{\Z}{{\mathbb Z}}

\newcommand{\bfX}{{\bf X}}
\newcommand{\cov}{\mathrm{Cov}}

\newcommand{\cC}{{\mathcal C}}
\newcommand{\cT}{{\mathcal T}}
\newcommand{\cS}{{\mathcal S}}
\newcommand{\cM}{{\mathcal M}}

\newcommand{\TT}{\mathcal{T}}
\newcommand{\bp}{\mathbf{p}}
\newcommand{\tX}{\widetilde{X}}

\newcommand{\be}{\mathbf{e}}

\renewcommand{\leq}{\leqslant}
\renewcommand{\geq}{\geqslant}

\begin{document}

\maketitle

\begin{abstract}
The Brownian sphere is a random metric space, 
homeomorphic to the two-dimensional sphere, 
which arises as the universal scaling limit of 
many types of random planar maps.
The direct construction of the Brownian sphere is 
via a continuous analogue of the 
Cori--Vauquelin--Schaeffer (CVS) bijection.
The CVS bijection maps labeled trees to planar maps, 
and the continuous version maps Aldous' continuum random tree 
with Brownian labels  
(the Brownian snake) to the Brownian sphere.
In this work, we describe the inverse of the continuous CVS bijection, 
by constructing the Brownian snake as a measurable function of the Brownian sphere.
Special care is needed to work with the orientation of the Brownian sphere.
\end{abstract}

\section{Introduction}\label{S_intro}

The {\it Cori--Vauquelin--Schaeffer (CVS) bijection} \cite{CV81,Sch98}
is a correspondence between labeled trees and rooted and pointed
quadrangulations of sphere, where by {\em pointed} we mean that an additional 
vertex, other than the root, has been distinguished. 
Given a (rooted and pointed) quadrangulation, the inverse CVS bijection 
yields a labeled plane tree, which is naturally embedded in the sphere, 
together with the map. 

A continuum analogue of the CVS bijection may be obtained by
associating a metric measure space $(X,d,\mu)$ with a pair of
real-valued functions $(f,g)$, where $f$ should be interpreted as the
``contour function'' of some $\R$-tree $\cT_f$, and $g$ as a label
function defined on $\cT_f$. 
When the pair of functions is random, with law given
by the {\it Brownian snake} \cite{LG93}, 
the resulting random metric space is called the 
{\it Brownian sphere} \cite{Mie13,LG13}.
More specifically, taking
$f$ to be a normalized Brownian excursion, 
$\cT_f$ is 
{\it Aldous' continuum
random tree (CRT)} \cite{Ald91_I,Ald91_II,Ald93_III}. 
This random $\R$-tree is given Brownian labels
$g$, 
which can be
interpreted as Brownian motion indexed by the CRT. 
The Brownian sphere 
is a singular, spherical 
metric space (almost surely homeomorphic 
to the 2-dimensional sphere \cite{LGP08}, 
but of Hausdorff dimension 4 \cite{legall07}), 
which
describes the scaling limit of many natural models of large planar
maps of the sphere, as was first established by 
Le Gall \cite{LG13} and Miermont \cite{Mie13}.

The purpose of this work is to invert the continuum CVS mapping,
so as to find the Brownian snake (and tree) in the Brownian sphere. 
While the Brownian sphere was first discovered as 
the scaling limit of random planar maps,
it has since been constructed by other means, 
with connections to {\it Liouville quantum gravity} \cite{MS20_I,MS21_II,MS21_III}.
Therefore, our results further emphasize that the Brownian tree and snake 
are fundamental objects in modern probability, which present themselves,  
regardless of the way in which the Brownian sphere is constructed. 

Our main result is the following. 

\begin{thm}\label{T_main}
  Let $(X,d,\mu)$ be the Brownian sphere, and let $x^0,x^1$ be two
  independent points in $X$, drawn 
  according to the measure
  $\mu$.  
  Then, almost surely, there exists a measurable function of
  $(X,d,\mu,x^0,x^1)$, that outputs
  an $\R$-tree $\cT$ 
  and a label function $Z:\cT\to\R$, 
  in such a way that:
\begin{enumerate}[nosep,label=(\roman*)]
  \item $\cT$ has the law of the CRT,
  \item $Z$ are Brownian labels on $\cT$, and
  \item the continuum CVS mapping applied to $(\cT,Z)$ recovers $(X,d,\mu)$.  
  \end{enumerate}
\end{thm}

The last point of the theorem requires some interpretation, 
because the continuum CVS mapping will formally be defined 
on pairs of functions encoding $(\cT,Z)$, which endow the tree $\cT$ 
with some extra structure; namely, with a certain ``planar order.'' 
See \cref{T:introduction-1} below for a formal version of \cref{T_main}.

\subsection{The construction}
\label{S_outline}

As we will see, the construction in \cref{T_main} is quite natural.  
In order to describe the inverse continuum CVS mapping, 
we use the geometric notion of a {\it cut locus}.
We recall that, if $X$ is a geodesic metric space, the cut locus of $X$ with
respect 
to a point $x\in X$ is the set of all points $z$ such that there are
at least two distinct geodesics between $z$ and $x$. 
Let $\cC(X,d,x)$ denote the cut locus of the Brownian sphere $(X,d)$ with respect to $x$.
For the Brownian sphere, Le Gall \cite{LG10} showed that, 
for $\mu$-almost every $x$, the cut locus $\cC$ has measure $0$, and is dense in $X$.
The inverse continuum CVS mapping is constructed in several steps, as outlined below:

\begin{enumerate}[nosep]
\item Given the doubly marked Brownian sphere $(X,d,\mu,x^0,x^1)$, let
  $\cC=\cC(X,d,x^1)$ be the cut locus with respect to $x^1$. 
\item For each point $y\in\cC$, define $Z(y) = d(y,x^1) -
  d(x^0,x^1)$. 
\item Observe that $(\cC,d)$ is almost surely (a.s.)\ path-connected and has no cycles.
  By local connectivity there is a topology on $\cC$ in which it is a dendrite
  embedded in the Brownian sphere.
\item Observe that $\cC$ does not include the leaves of the dendrite, however, the labels 
$Z$ extend continuously to
the closure $\cT$ of $\cC$ with respect to a local topology, intrinsic to $\cC$. 
\item Define a metric on $\cT$ so that $Z:\cT\to\R$ is a Brownian
  motion indexed by $\cT$. 
  To find the distance, note that $d_\cT(u,v)$ is the 
  quadratic variation of the function $Z$ along the unique path between $u$ and $v$ in $\cT$. 
\item The orientation of the Brownian sphere  
  induces an order on $\cT$, making it a plane tree, 
  and thereby canonically encoded by a contour process. 
\end{enumerate}

These steps will be justified in \cref{sec:proof-theor-refs}, 
where we will also address the delicate question of the measurability of the resulting mapping. 
Let us discuss a little more the role of the marked points and of the orientation in the construction.

\subsection{The marked points}

The two marked points $x^0,x^1$ indeed play a crucial role in the recovery of the label process. 
Roughly speaking, we will use one marked point to determine its ``time origin,'' 
and the other marked point for its values.

The roles of $x^0$ and $x^1$ are  very different in the mapping that 
we construct.
For a given space $(X,d,\mu,x^0,x^1)$ as in \cref{T_main}, 
a change of $x^0$ amounts to a re-rooting of the tree $\cT$. 
A change in $x^1$, on the other hand,  modifies the tree and process $Z$ profoundly.
Roughly, the tree is cut into countably many trees, 
which are then reconnected in a different manner, 
while maintaining the increments of $Z$ on each part. 

We note that the Brownian sphere usually comes 
equipped with a distinguished point $\rho$, called the {\it root}.
By Le Gall's \cite{LG10} {\it re-rooting property}, 
$\rho$ is equal in distribution to a random point $x$ drawn according to the measure $\mu$.
The two marked points $x^0,x^1$ in our construction are independent, 
with the same distribution as $\rho$.

\subsection{Choosing an orientation}

The choice of an orientation for the Brownian sphere also plays a crucial role in our construction.  
The Brownian sphere is homeomorphic to a sphere, and thus orientable. 
Previous works, however, have mostly avoided explicitly discussing the issue of orientation.

Both constructions of the Brownian sphere, as a quotient of the CRT or via Liouville quantum gravity, 
give rise to a natural orientation of the sphere.
However, if we are given the Brownian sphere as an unoriented metric space, 
there are two possible orientations.
Reversing the orientation of the sphere corresponds to a reflection of the CRT.

It is worth explaining what we mean when discussing the Brownian sphere as an oriented space. 
It is not a smooth manifold, so some standard definitions fail.
There are several possible equivalent definitions for the orientation of a topological sphere $X$:
\begin{enumerate}[nosep]
\item For each simple loop we may regard one side as the interior and the other as the exterior, making the choice continuously on the space of loops in $X$.
\item For each collection of pairwise disjoint paths emanating from $x\in X$, 
we can determine a cyclic order, again making the choice continuously on the space of such ``stars.''
\item Observe that the homology group $H_2(X,\Z)$ is isomorphic to $\Z$, and select one of the generators as $+1$ and the other as $-1$.
\end{enumerate}
Of these, the first and third extend naturally to higher dimensions. The first two are ``local'' and generalize to topologies other than the sphere.
In the case of a topological sphere, these definitions are all equivalent.

To deal with the fact that the Brownian sphere has been defined as an {\it unoriented} metric space, there are two approaches that one might take.
One could consider the Brownian sphere as an oriented metric space to begin with.
This would require revisiting the definition of the Brownian sphere, and to view it as a random element of a space of oriented metric spheres. Working in such an abstract space would create some undesirable technicalities. 
Specifically, a number of steps in the construction of the Brownian sphere would need to be revisited to show that all steps work with oriented spheres.
For example,  Le Gall and Paulin \cite{LGP08} used Moore's theorem to show that the Brownian sphere is homeomorphic to the sphere.
The usual form of Moore's theorem does not provide an orientation for the quotient space, though one can verify that there is a natural orientation.

Another approach, which we will follow, is to simply consider the Brownian sphere as a metric space, and assign to it a random orientation based on an independent variable $\epsilon\in\{\pm1\}$, to account for the two possible orientations. This is conceptually simpler, but poses some measurability issues that will need to be addressed.

\subsection{Other Brownian surfaces}

The Brownian sphere is only one element in a family of compact \cite{BeMi17,BeMi22} or non-compact \cite{CLG14,BaMiRa19}
Brownian random surfaces, including the Brownian plane, disc, and any compact orientable surface with or without boundary. 
Although we did not work out the details, we are confident that the methods and results of this paper generalize to these objects with little change, and similar inverse bijections to that of  \cref{T_main,T:introduction-1} can be derived in these settings, because Brownian random surfaces all have similar constructions in terms of labeled tree-like structures. 

For instance, the {\it Brownian plane} \cite{CLG14}, which is a random pointed metric measure space $(X,d,\mu,x^0)$ that is a.s.\ homeomorphic to $\R^2$, is obtained if the Brownian excursion used to generate the CRT is replaced by a two-sided, infinite Brownian motion encoding the infinite, self-similar CRT, see e.g.\ \cite{BaMiRa19}. In this setting, the second marked point $x^1$ can be taken to be at infinity.
The labels $Z$ are well-defined, since there is a unique Busemann function $B_{x^0}(x) = d(x,\infty)-d(x^0,\infty)$ on the Brownian plane, and the cut locus with respect to infinity is the set of points from which one can start multiple geodesics rays. From these basic facts, one should be able apply the methods of \cref{sec:proof-theor-refs} with little change. 
 In this situation, however, the ``orientation'' random variable $\epsilon_h$ appearing below in \eqref{eq:epsilon-h} would need to be defined in some other way, but fortunately, there is some flexibility in its choice, see the last remark of \cref{sec:the-role-of-epsilon}.

\subsection{Organization  of the paper}
The rest of the paper consists of three sections. 
In \cref{sec:rand-snak-brown}, we review the basic constructions of the Brownian tree, snake and sphere, and properly define the continuum Cori-Vauquelin-Schaeffer mapping that we want to invert. 
\cref{sec:cont-inverse-cvs} states and discusses \cref{T:introduction-1}, a more precise version of  \cref{T_main}, and shows how the latter follows from the former. 
\cref{sec:proof-theor-refs} provides the proof of \cref{T:introduction-1}, following the steps  
in \cref{S_outline}.

\subsection{Acknowledgments}

We thank Jean-Fran\c{c}ois Le Gall for useful discussions. 
The authors would like to thank the Isaac Newton Institute for Mathematical Sciences, Cambridge,  for support and hospitality during the program  Random geometry, where initial work on this paper was undertaken. 
This work was supported by EPSRC grant EP/K032208/1.
OA is supported in part by NSERC.

\section{Random snakes and the Brownian sphere}
\label{sec:rand-snak-brown}

We first review the definitions of some of the known objects we will work with, namely, the Brownian tree, snake and sphere, and the properties that will be needed in the proof of 
\cref{T:introduction-1} below.

\subsection{Deterministic encodings}
\label{sec:determ-encod}

Let us first give a general discussion on some metric structures encoded by deterministic functions.

\subsubsection{Trees coded by functions}
\label{sec:trees-functions}

If $f:[0,1]\to \mathbb{R}$ is a continuous function 
such that $f(0)=f(1)=0$, we define
$$d_f(s,s')=f(s)+f(s')-2\max\left(\inf_{[s\wedge s',s\vee
    s']}f,\inf_{[0,s\wedge s']\cup [s\vee s',1]}f\right)\, ,$$
which defines a pseudometric on $[0,1]$. This means that it is a symmetric
function that 
satisfies the triangle inequality, and vanishes on the diagonal.
In fact, we can alternatively view $f$ as a function $f:\mathbb{S}^1\to \R$, where the unit circle $\mathbb{S}^1$ is identified with $\R/\Z$, and, with this identification, the expression for $d_f$ simplifies to 
$$d_f(s,s')=f(s)+f(s')-2\max\left(\inf_{[s,s']_\circ}f,\inf_{[s',s]_\circ}f\right)\, ,$$
where $[s,s']_\circ$ and $[s',s]_\circ$ are the two naturally oriented arcs of $\mathbb{S}^1$ between $s$ and $s'$.  

The quotient
$\cT_f=[0,1]/\{d_f=0\}$, endowed with the distance inherited from
$d_f$, is an $\R$-tree, and for $a,b\in \TT_f$, we denote by $[[a,b]]_f$ the unique geodesic
segment between $a$ and $b$ in $(\cT_f,d_f)$. 

 It naturally carries a distinguished point
$\rho_f=p_f(\mathrm{argmin}\ f)$, where $p_f:[0,1]\to \TT_f$ is the
quotient mapping. Note $d_f(s,t)=0$ whenever $f$ attains its infimum at $s$ and $t$, so $\rho_f$ is well defined. 

We call $(\cT_f,d_f,\rho_f)$ the 
{\it rooted tree coded by $f$}.

\subsubsection{The mating of trees metric}\label{sec:mating-trees-metric}

Let us denote by $\cS$ the space of pairs $h=(f,g)$ of 
continuous functions $f,g:[0,1]\to \R$, taking value $0$
at $0$ and $1$. 
We associate with $f$ and $g$ the trees $(\cT_f,d_f)$ and $(\cT_g,d_g)$ as above\footnote{In our applications below, starting
  with \cref{sec:brownian-tree-snake}, we will
  require the further property that $(f,g)$ encodes a ``snake
  trajectory.'' This means that $g$ can be interpreted as a label function
on the tree $\cT_f$, in the sense that $g(s)=g(t)$ for any $s,t$ such
that $d_f(s,t)=0$. However, we
stress that the discussion of \cref{sec:determ-encod} holds in full generality.}. 

For $h\in \cS$, the {\it mating of trees pseudometric} 
is defined, for every $s,t$ in $[0,1]$, by 
$$d_h(s,t)=\inf\left\{\sum_{i=1}^k d_g(s_i,t_i):
\begin{split}
    k\geq 1, s=s_1,t_1,s_2,t_2,\ldots,s_k,t_k=t\\
    d_f(t_j,s_{j+1})=0,1\leq j\leq k-1
  \end{split}\right\}\, .$$
The quotient $X_h=[0,1]/\{d_h=0\}$, endowed with the distance  inherited from
$d_h$, is called the {\it mating of trees coded by $h$}. Informally, it is
obtained by ``wrapping'' the tree $(\cT_g,d_g)$ around the tree
$\cT_f$. Note that the construction is not at all symmetric in
the two trees. In particular, it depends strongly on $d_g$, but
depends on $f$ only through the set $\{d_f=0\}$, that is, on the topological space $\TT_f$. 

As a consequence of the definition of the quotient metric $d_h$, there is a natural commutative diagram: 
\begin{equation}\label{eq:commut}
\begin{tikzcd}
{[}0,1{]} \arrow{r}{p_f} \arrow{d}{p_g} \arrow{rd}{p_h}
&\cT_f \arrow{d}{\bp_f}\\
\cT_g \arrow{r}{\bp_g} & X_h\, 
\end{tikzcd}
\end{equation}
All projection mappings
$p_f,\bp_f,p_g,\bp_g,p_h$ are
continuous. 
Therefore, $X_h$ can be seen as the quotient of either
$\TT_f$ or $\TT_g$, with canonical projections
$\bp_f,\bp_g$. 

The space $X_h$ naturally carries two distinguished points 
$$x_h^0=p_h(\mathrm{argmin}\ f)=\bp_f(\rho_f)\quad \mbox{ and }\quad
x_h^1=p_h(\mathrm{argmin}\ g)=\bp_g(\rho_g)\, ,$$ where, we recall that 
$\rho_f,\rho_g$ are the roots of the trees $\TT_f,\TT_g$. 

We also define 
the image measure 
$\mu_h=(p_h)_*\mathrm{Leb}_{[0,1]}$, which we call the
{\it area measure on $X_h$}.

\subsubsection{Gromov--Hausdorff-type
  spaces}\label{sec:grom-hausd-type}

In order to consider random versions of 
the metric space constructions
discussed above, with $h\in\cS$ random, 
we will need to introduce Gromov--Hausdorff-type spaces.

Suppose that $(X,d_X,\mu_X,(a_i)_{1\leq i\leq n})$ and
$(Y,d_Y,\mu_Y,(b_i)_{1\leq i\leq n})$ are two compact metric measure spaces
with $n\geq 0$ distinguished points.
We call them {\it isometric}
if there exists an isometry
$\varphi:(X,d_X)\to (Y,d_Y)$ such that $\mu_Y=\varphi_*\mu_X$ and
$\varphi(a_i)=b_i$ for $1\leq i\leq n$. 
We let 
$[X,d_X,\mu_X,(a_i)_{1\leq i\leq n}]$
denote the isometry
class of $(X,d_X,\mu_X,(a_i)_{1\leq i\leq n})$. 
We endow the set  
$m\mathcal{M}^{n\bullet}$ 
of all such isometry classes with the
Gromov--Hausdorff--Prokhorov topology. 
We write $\bfX^{n\bullet}$
to denote a generic element of $m\mathcal{M}^{n\bullet}$, 
to emphasize that it has $n$ marked points. 
When $n=0$, we
write  $m\mathcal{M}=m\mathcal{M}^{0\bullet}$ and $\bfX=\bfX^{0\bullet}$. 
In \cref{sec:role-measure} below, 
we will also consider the Gromov--Hausdorff
space $\mathcal{M}$ of isometry classes $[X,d]$ of compact metric spaces, 
endowed with the Gromov--Hausdorff topology. 

In particular, for every $h\in \cS$, the isometry class $[X_h,d_h,\mu_h,x_h^0,x_h^1]$ 
is an element of $m\mathcal{M}^{2\bullet}$, which we denote by $\bfX_h^{2\bullet}$. 
We also let $\bfX_h=[X_h,d_h,\mu_h]$ be its unmarked version.

\subsection{The continuum CVS mapping $\psi$}\label{sec:cont-cvs-mapp}

Finally, using the above definitions, 
we obtain a mapping
\begin{equation}
  \label{eq:1}
\psi:
  \begin{array}{ccl}
\mathcal{S}&\longrightarrow& m\mathcal{M}^{2\bullet}\\
    h&\longmapsto & \bfX_h^{2\bullet}\, , 
  \end{array}
\end{equation}
which we call the {\it continuum CVS mapping}.

\begin{prop}
  \label{sec:grom-hausd-type-1} The continuum CVS mapping is
  measurable. 
\end{prop}

\begin{proof}
Let $\mathcal{C}_m$ be the space of continuous functions from
    $[0,1]^2$ to $\R_+$ that are pseudometrics. 
  Consider the mapping
$$ \psi_1: \begin{array}{ccl}
\mathcal{S}&\longrightarrow&  \mathcal{C}_m\times [0,1]\times [0,1]\\
    h=(f,g)&\longmapsto & ((d_h(s,t),0\leq s,t\leq 1),s_*(f),s_*(g))\, , 
    \end{array}$$
where $s_*(f)$ (resp.\
$s_*(g)$) is the first times that
$f$ (resp.\ $g$) attains its minimum. 
 We already mentioned
    that $d_h$ is a pseudometric. 
The fact that it is
    continuous can be seen by writing
    $$|d_h(s,t)-d_h(s',t')|\leq d_h(s,s')+d_h(t,t')\leq
    d_g(s,s')+d_g(t,t')\, ,$$
    which goes to $0$ as $(s',t')\to (s,t)$. 
    Moreover, for a fixed integer $k\geq 1$, for $\delta>0$ and for $s,t\in [0,1]$, the quantities
    $$d^{(k,\delta)}_h(s,t)=\inf\left\{\sum_{i=1}^k d_g(s_i,t_i):\begin{split}
    s=s_1,t_1,s_2,t_2,\ldots,s_k,t_k=t\\
    d_f(t_j,s_{j+1})<\delta,1\leq j\leq k-1
  \end{split}\right\}$$
are measurable functions of $h$ (note that the infimum can be taken
over choices of $t_1,s_2,t_2,\ldots,s_k$ in $\mathbb{Q}$), so that the same holds for
$d_h(s,t)$, since it equals $\lim_{k\to\infty}\lim_{\delta\downarrow
  0}d^{(k,\delta)}_h(s,t)$. The measurability of $\psi_1$ follows. 

Finally, consider the mapping
$$ \psi_2: \begin{array}{ccl}
 \mathcal{C}_m\times [0,1]\times[0,1]&\longrightarrow&m\mathcal{M}^{2\bullet}\\
     (d,s,t)&\longmapsto &[[0,1]/\{d=0\},d,p_*\mathrm{Leb},p(s),p(t)]
    \end{array}\,,$$
where for $d\in \mathcal{C}_m$, the mapping $p:[0,1]\to [0,1]/\{d=0\}$
is the canonical projection map. This mapping is easily seen to be
continuous. Altogether, we obtain the result, noting that $\psi=\psi_2\circ\psi_1$. 
  \end{proof}

The following lemma is a simple but important observation. 
For $h\in \cS$, we denote by $R(h)$ the reversed path
$(f(1-\cdot),g(1-\cdot))$.

\begin{lemma}\label{L:brown-sphere-metr}
 For every $h\in \mathcal{S}$, the spaces 
 $\bfX_h^{2\bullet}$ and $\bfX_{R(h)}^{2\bullet}$ are
 equal as elements of $m\mathcal{M}^{2\bullet}$. 
\end{lemma}

The proof is immediate, noting that $s\in [0,1]\mapsto 1-s$ 
induces an isometry between $(X_h,d_h,\mu_h,x_h^0,x_h^1)$ 
and $(X_{R(h)},d_{R(h)},\mu_{R(h)},x_{R(h)}^0,x_{R(h)}^1)$. 
We remark that if one considers the quotient spaces as oriented spheres, 
then $\bfX_{R(h)}^{2\bullet}$ is the reflection of $\bfX_{h}^{2\bullet}$, 
i.e. the same space with the opposite orientation.

\subsection{The Brownian tree, snake and sphere}\label{sec:brownian-tree-snake}

The Brownian tree, snake and sphere
are obtained by applying the 
above constructions to 
certain  natural random
functions. 

Let $(\be_t,0\leq t\leq 1)$ be a normalized Brownian excursion. 
The tree $( \cT_\be,d_\be)$ that it encodes is called Aldous'
{\it continuum random tree} (CRT). The Brownian excursion is
nonnegative, so its root is simply 
$\rho_\be=p_\be(0)$.

Next, conditionally given $\be$, define a continuous centered Gaussian process
$(Z_t,0\leq t\leq 1)$ with covariance function
$$\cov(Z_s,Z_t)=\inf_{s\wedge t\leq u\leq s\vee t}\be_{u}.$$
The fact that such a process has a continuous version is well known. 
Moreover, almost surely, $Z_s=Z_t$ for
every $s,t$ such that $d_\be(s,t)=0$, so that $Z$ can be seen as a
process indexed by $\cT_\be$. 
Indeed, this process can be interpreted as  Brownian motion indexed by
$\cT_\be$, with value $0$ at $\rho_\be$, and with independent
increments over disjoint paths. 
The pair $W=(\be,Z)$ is a random element of $\cS$, which we call the 
{\em Brownian snake}\footnote{In fact, what is usually called the {\em Brownian snake}
 is the 
path-valued Markov process encoding the trajectories of $Z$ from
the root of $\cT_\be$, as one follows the contour of the tree. The
process $\be$ is called the {\it lifetime process} of the snake, and $Z$ the
{\it head of the snake}. However, this path-valued process
can
be recovered as a continuous function of  $(\be,Z)$.}. Its law is denoted by $\P_{\mathrm{Snake}}$.  

Finally, applying the continuum CVS mapping 
$\psi:h\mapsto \bfX_h^{2\bullet}$ to the Brownian snake, 
we obtain the
(doubly marked) {\it Brownian sphere} $\bfX_W^{2\bullet}$, 
with values in $m\mathcal{M}^{2\bullet}$. 
We denote its law by $\mathbb{P}^{2\bullet}_{\mathrm{Sphere}}=\psi_*\P_{\mathrm{Snake}}$.
We also denote by $\mathbb{P}_{\mathrm{Sphere}}$ 
the law of the \emph{unmarked} Brownian sphere $\bfX_W$, with values in $m\mathcal{M}$.

\section{The inverse continuum CVS mapping $\phi$}
\label{sec:cont-inverse-cvs}

Our goal is to invert the continuum CVS mapping, 
that is, to find a measurable mapping $m\cM^{2\bullet}\to\cS$
that is an (almost sure) inverse of $\psi$ 

One possible approach is as follows.
Suppose we find a Borel subset $A\subset \cS$ such that
$\P_{\mathrm{Snake}}(\cS\setminus A)=0$ and  $\psi|_A$ is injective.
Then, by a general theorem of Lusin and Souslin on measurable functions on
Polish spaces (see, e.g., Kechris \cite[Corollary 15.2]{kechris95}), 
it would follow that $\psi(A) \subset m\cM^{2\bullet}$ is a Borel set, 
and that the inverse $\phi:\psi(A)\to A$ is Borel measurable. 

This approach, however, 
has two major drawbacks.
First, it gives an abstract inverse function, without telling us anything about its nature or how to compute it.
Secondly, this approach requires the continuum CVS mapping to be almost
everywhere injective, which it is {\it not,} due to 
\cref{L:brown-sphere-metr}.
Applying the reflection $R$, while preserving the planar tree structure and Brownian labels, has the effect of reversing the orientation of the Brownian sphere (as will be further discussed in \cref{sec:recovering-z}). 
However, $m\cM^{2\bullet}$ does not include information about the orientation of the sphere. 
Thus, as we will state and prove rigorously below (see \cref{T:introduction-1}), 
the continuum CVS mapping is actually 2-to-1
on a set of full probability for the law $\P_{\mathrm{Snake}}$.

This is also the case for discrete maps, 
if one does not keep track of their orientation.
This has not been an issue in the past since discrete maps 
have always been considered as oriented (not just orientable), 
so that a map and its reflection are considered as distinct objects. 
However, previous works on the Brownian sphere have opted 
to omit the orientation when taking the limit, 
so that a map and its reflection are considered identical.

To overcome this issue, we work in an extended space
$m\mathcal{M}^{2\bullet}\times \{\pm1\}$. 
The coordinate in $\{\pm1\}$ is used to determine the {\it
  orientation} of the Brownian sphere. 
As one might expect, this orientation turns out to be uniformly random
(see \cref{T:introduction-3}). 
However, as we will see, describing 
this choice in a {\it measurable} way requires some care. 

For $h=(f,g)\in \cS$, recall that $s_*(g)$ is the first time that  $g$
attains its minimum. We put 
\begin{equation}
  \label{eq:epsilon-h}
  \epsilon_h= \mathrm{sgn}(1-2s_*(g)) = 
  \begin{cases}
    1& \text{ if }s_*(g)\leq 1/2 \\
    -1 & \text{ if }s_*(g)>1/2\, .
  \end{cases}
\end{equation}
We then extend the mapping $\psi$ by  
\begin{equation}
  \label{eq:7}
  \overline{\psi}:
  \begin{array}{ccl}
\mathcal{S}&\longrightarrow &m\mathcal{M}^{2\bullet}\times \{\pm1\}\\
    h&\longmapsto &(\psi(h),\epsilon_h)\, .
  \end{array}
\end{equation}

When $W=(\be,Z)$ follows the law of the Brownian
snake, the infimum of $Z$ is a.s.\ attained at the unique time $s_*=s_*(Z)$, as
shown in \cite[\S2.3]{LGW06}. Since $Z$ has cyclically
exchangeable increments, this implies that $s_*$ is uniformly
distributed in $[0,1]$, see \cite[Corollary 3.1]{ChUB15}. 
In particular, $\epsilon_W$ follows the Rademacher
distribution (equal to $\pm1$ with equal probability). 

We may finally state our main result in a precise form.  

\begin{thm}
  \label{T:introduction-1}
There exists a Borel mapping $\phi:m\mathcal{M}^{2\bullet}\times
  \{\pm1\}\to \mathcal{S}$ which is an a.s.\ inverse of
  $\overline{\psi}$, in the sense that: 
\begin{enumerate}[nosep,label=(\roman*)]
  \item 
  $\mathbb{P}^{2\bullet}_{\mathrm{Sphere}}(d\bfX^{2\bullet})$-a.s., and for
  every $\epsilon\in \{\pm1\}$, 
  $$\overline{\psi}\circ\phi(\bfX^{2\bullet},\epsilon)
  =(\bfX^{2\bullet},\epsilon)\,
  ,\qquad \text{ and}$$
\item $\mathbb{P}_{\mathrm{Snake}}(d h)$-a.s., 
  $$\phi\circ \overline{\psi}(h)=h\, . $$
\end{enumerate}
\end{thm}

This result immediately yields that, 
$\mathbb{P}^{2\bullet}_{\mathrm{Sphere}}(d\bfX^{2\bullet})$-a.s., 
for $\epsilon\in \{\pm1\}$, 
\[
\psi\circ\phi(\bfX^{2\bullet},\epsilon)=\bfX^{2\bullet}\,,
\]  
which justifies our interpretation of $\phi$ as
an inverse of the continuum CVS mapping $\psi$. 
See also \eqref{eq:6} below for composing $\phi$ and $\psi$ in the reverse order.

In the following subsections, we will explain why the above statement
implies \cref{T_main}, which will 
also allow us to
clarify the role of the parameter $\epsilon$ and of the two marked
points that appear in the input of the mapping $\phi$.

\subsection{The role of $\epsilon$}\label{sec:the-role-of-epsilon}

Let us first discuss the meaning of the parameter $\epsilon\in\{\pm1\}$. 
We claim that, $\mathbb{P}_{\mathrm{Snake}}(dh)$-a.s., 
\begin{equation}
  \label{eq:3}
  \overline{\psi}(R(h))=(\psi(h),-\epsilon_h),  
\end{equation}
recalling that $R(h)$ is time-reversal of $h$. 
Indeed, as already mentioned,  
$\psi(h)$ and $\psi(R(h))$ are
equal as elements of $m\mathcal{M}^{2\bullet}$.
Also, $\epsilon_{R(h)}=-\epsilon_h$, 
$\mathbb{P}_{\mathrm{Snake}}(dh)$-a.s., because the
law of $s_*(g)$ under $\P_{\mathrm{Snake}}$ is diffuse, so that the probability
$s_*(g)=1/2$ (and hence that $\epsilon_{R(h)}=\epsilon_h$) is zero.  
In particular, we deduce the following two properties. First,
$\mathbb{P}_{\mathrm{Snake}}(dh)$-a.s., for $\epsilon \in \{\pm1\}$, 
\begin{equation}
  \label{eq:6}
  \phi\left( \psi(h),\epsilon\right)\in \{h,R(h)\}\, .
\end{equation}
Second, $\mathbb{P}^{2\bullet}_{\mathrm{Sphere}}(d\bfX^{2\bullet})$-a.s.,
\begin{equation}
  \label{eq:5}
  \phi(\bfX^{2\bullet},-\epsilon)=R(\phi(\bfX^{2\bullet},\epsilon))\,
  ,\qquad \epsilon \in\{\pm1\}\, .
\end{equation}
These properties show that $\epsilon$ determines the
``orientation'' of $h$. As we will see in the course of the proof
of  \cref{T:introduction-1}, $\epsilon$ 
allows one to distinguish $\mathbb{P}_{\mathrm{Snake}}$-almost surely $h$ from $R(h)$. 
Moreover, $\epsilon$ determines an
orientation of $\psi(h)$, which under $\P_h$ is a.s.\ a topological sphere, and this
orientation is reversed by changing $\epsilon$ to $-\epsilon$. See
\cref{sec:recovering-z} for more details.  

Next, we describe the joint law of
$(\psi(h),\epsilon_h)$ under $\P_{\mathrm{Snake}}$. 

\begin{prop}
  \label{T:introduction-3}
Under 
  $\mathbb{P}_{\mathrm{Snake}}(dh)$, the random variables $\psi(h)$
  and $\epsilon_h$ are independent, and  
  the latter has the Rademacher distribution.   
\end{prop}

\begin{proof}
  We use the well-known (and easy to check) 
  fact that $R(h)$ has same
  distribution as $h$ under $\mathbb{P}_{\mathrm{Snake}}$. Therefore, \eqref{eq:3}
  implies that $(\psi(h),\epsilon_h)$ and $(\psi(h),-\epsilon_h)$ have
  same distribution under $\mathbb{P}_{\mathrm{Snake}}$. Hence, the random variables
  $\pm \epsilon_h$ have the same law conditionally given
  $\psi(h)$. The result follows. 
\end{proof}

The above proposition and the discussion just before its statement are
in line with the intuitive fact that, 
conditionally given the (orientable, but not oriented) random sphere
$\bfX_W^{2\bullet}$ considered in \cref{sec:brownian-tree-snake}, 
the choice of orientation is
determined by an independent fair coin flip. On rigorous grounds, the mapping $\phi$
allows this choice to be done in a measurable way. 

As an immediate consequence, we have the following. 
Recall that $\mathbb{P}^{2\bullet}_{\mathrm{Sphere}}$ is
the image measure $\psi_*\mathbb{P}_{\mathrm{Snake}}$. 

\begin{corollary}
  \label{T:introduction-2}
  The measure $\mathbb{P}_{\mathrm{Snake}}$ is the image measure of
  $\mathbb{P}^{2\bullet}_{\mathrm{Sphere}}\otimes (\delta_{-1}+\delta_1)/2$ 
  under the mapping $\phi$. 
\end{corollary}

\begin{remark}
    We mention that the choice of the random variable $\epsilon_h$ is not  canonical. 
    \cref{T:introduction-1} would remain true 
    (although for a possibly different inverse function $\phi$) 
    if we had adopted some other definition than \eqref{eq:epsilon-h}, 
    as long as it  almost surely differentiates $h$ from $R(h)$. 
    However, for a choice of $\epsilon_h$ 
    not satisfying the additional property~\eqref{eq:3}, \cref{T:introduction-2} would not be true.
\end{remark}

\subsection{The role of the marked points}
\label{sec:roles-marked-points}

The role of $\epsilon$ being now (partially) clarified, 
let us address the role of the distinguished points
$x_W^0,x_W^1$. Recall the notation $\mathbb{P}_{\mathrm{Sphere}}$ for the law of the
\emph{unmarked} Brownian sphere.

By the re-rooting property 
\cite{LG10}, conditionally given 
$[X,d,\mu]$ with law $\mathbb{P}_{\mathrm{Sphere}}$, if $y^0$ and $y^1$ are two
independent random points in $X$ with law $\mu$,  then
$[X,d,\mu,y^0,y^1]$ has law $\mathbb{P}_{\mathrm{Sphere}}^{2\bullet}$. 
In other words, the
points $x_W^0,x_W^1$ play the role of two independent marked points, 
distributed according to the area measure\footnote{Strictly speaking, this discussion is not completely accurate
  since we are working with isometry classes of spaces, and not with
  fixed space. A more precise
statement would be in terms of Markov kernels $m\mathcal{M}\to
m\mathcal{M}^{2\bullet}$; see \cite[Section 6.5]{miermont09}. 
However, for ease of exposition, 
we will use to this
slightly informal formulation.}. 

Summarizing the discussion of this and the preceding sections, we
arrive at the following statement, which is a more precise
version of  \cref{T_main}. 

\begin{corollary}
\label{T:introduction-4}
Let $\bfX=[X,d,\mu]$ have law $\mathbb{P}_{\mathrm{Sphere}}$. Conditionally
given $\bfX$, let $x^0,x^1$ be two independent
random points in $X$ with law $\mu$.
Let 
$W_\pm=\phi(\bfX^{2\bullet},\pm 1)$, 
where $\bfX^{2\bullet}=[X,d,\mu,x^0,x^1]$. 
Then, almost surely,   
$R(W_+)=W_-$ and 
$\psi(W_+)=\psi(W_-)=\bfX^{2\bullet}$. 
Moreover, if $\sigma\in \{+,-\}$ is itself random, independent of
$\bfX^{2\bullet}$, and uniformly distributed, then
$W_\sigma$ has law $\P_{\mathrm{Snake}}$. 
\end{corollary}

\subsection{The role of the measure}\label{sec:role-measure}

Finally, let us discuss the role of the area measure $\mu$.
As shown by Le Gall \cite{LG22}, this measure is almost surely determined by the
metric structure, in the following sense. 
Fix the gauge function $H(r)=r^4\log \log
(1/r)$.
Then there exists a constant $c\in (0,\infty)$ such that
$\mathbb{P}_{\mathrm{Snake}}(dh)$-a.s., $\mu_h$ is equal to $c$ times the
$H$-Hausdorff measure on the space $(X_h,d_h)$. 
Hence the mapping $[X,d,\mu]\mapsto[X,d]$ from the 
Gromov--Hausdorff--Prokhorov space $m\mathcal{M}$ 
to the Gromov--Hausdorff space $\mathcal{M}$ is injective on a Borel set of full measure. 

Therefore, by the Lusin--Souslin theorem (already mentioned in \cref{sec:cont-inverse-cvs} above),  
there exists a Borel mapping $\chi: \mathcal{M}\to m\mathcal{M}$ such that $\chi([X,d])=[X,d,\mu]$, and
$\mathbb{P}^\circ_{\mathrm{Sphere}}(d[X,d])$-a.s., 
where $\mathbb{P}^\circ_{\mathrm{Sphere}}$ 
is the law of the unmeasured space $[X_h,d_h]$ under $\mathbb{P}_{\mathrm{Snake}}(dh)$. 
Given these observations, it is possible to state a version
of \cref{T:introduction-4}, with $\mathbb{P}_{\mathrm{Sphere}}$ replaced
by $\mathbb{P}^\circ_{\mathrm{Sphere}}$. We omit the details.

\section{Proof of \cref{T:introduction-1}}\label{sec:proof-theor-refs}

Recall that, by virtue of the Lusin--Souslin theorem, 
it suffices to prove that, for
$\mathbb{P}_{\mathrm{Snake}}$-a.e.\ $h=(f,g)\in \mathcal{S}$, 
the functions $f$ and
$g$ can be 
expressed as a function of $\bfX_h^{2\bullet}$ and $\epsilon_h$.

From now on, we let $W=(\be,Z)$ 
be a random variable with law $\P_{\mathrm{Snake}}$
defined on some probability space, and our goal is to 
show that $W$ is determined by $\bfX_W^{2\bullet},\epsilon_W$ 
on a set of full probability.  
We prove this by following the steps listed in  \cref{S_outline}.

\subsection{The two intertwined trees\label{sec:two-intertw-trees}}

If $(\TT,d)$ is an $\R$-tree and $x\in \TT$, we let $\mathrm{deg}_\TT(x)\in \{1,2,\ldots\}\cup\{\infty\}$ be the number of connected components of $\TT\setminus\{x\}$. The sets 
$$\mathrm{Leaf}(\TT)=\{x\in \TT:\mathrm{deg}_\TT(x)=1\} \quad \mbox{ and }\quad\mathrm{Skel}(\TT)=\TT\setminus\mathrm{Leaf}(\TT)$$ 
are called the set of {\it leaves} and the {\it skeleton} of $\TT$. We will use the well-known fact that if $f:[0,1]\to \R$ is a continuous function with $f(0)=f(1)=0$, then the points of $\mathrm{Skel}(\TT_f)$ are exactly those which have at least two preimages in $[0,1)$ under $p_f$. 

Our proof of \cref{T:introduction-1} will 
depend a great deal on the fact that the images of the trees
$\mathrm{Skel}(\TT_\be)$ and $\mathrm{Skel}(\TT_Z)$ under the projections
$\bp_\be$ and $\bp_Z$ are geometric loci in the Brownian
sphere. Namely, almost surely, 
 the first one is the cut locus of $(X_W,d_W)$ with
respect to the point $x_W^1$, and the second is the set of relative
interiors of geodesics towards $x_W^1$. These properties are already known, and summarized as follows.

\begin{lemma}
    \label{lem:geominterpr}
    The following properties hold almost surely. 
\begin{enumerate}[nosep,label=(\roman*)]
\item The restrictions of $\bp_\be$ and $\bp_Z$ to
$\mathrm{Skel}(\TT_\be)$ and   $\mathrm{Skel}(\TT_Z)$ are homeomorphisms onto
  their images, which we denote by $C_W$ and
  $\Gamma_W$, respectively. 
    
\item The set $C_W$ is the {\it cut locus} of $(X_W,d_W)$ with respect to $x_W^1$, which consists of the points $x\in X_W$ such that there
    exists at least two distinct geodesic paths from $x$ to $x_W^1$.
    
\item The set $\Gamma_W$ 
(of points ``inside'' geodesics to $x_W^1$) 
consists of the points $x\in X_W$ such that
      there exists a geodesic segment in $(X_W,d_W)$ that contains
      $x$, and whose extremities are $x_W^1$ and some other point
      $y\neq x$. 
    
\item It holds that $C_W\cap\Gamma_W=\varnothing$, and moreover, for every $x\in X_W\setminus \Gamma_W$ (resp.\ $x\in X_W\setminus C_W$), $x$ has a unique preimage under $\bp_\be$ (resp.\ $\bp_Z$).
    \end{enumerate}
\end{lemma}

For convenience, let us also 
define  
\[
\tX_W= X_W\backslash(C_W\cup \Gamma_W).
\] 
See \cref{fig:twospheres} for an illustration
of the above results. 

\begin{proof}
Points (i)--(iii) are proved in \cite[Proposition 3.1]{LG10} and \cite[Theorem 6.3.3]{miermont14}. More precisely, it holds that the geodesic segments to $x_W^1$ in $(X_W,d_W)$ are exactly of the form $\bp_Z([[\rho_Z,a]]_Z)$, where $a\in \TT_Z$, which are called the \emph{simple geodesics} in \cite{LG10}, and one has $d_W(x_W^1,\bp_Z(a))=d_Z(\rho_Z,a)$. 

Point (iv) is a consequence of \cite[Theorem 3.4]{legall07} and \cite[Lemma 3.2]{LGP08}. Let us give a bit more details on this last point. The two aforementioned results can be reformulated as follows: 
for every $s\neq t$ in $(0,1)$, one has $d_W(s,t)=0$ if and only if $d_\be(s,t)=0$ or $d_Z(s,t)=0$, these two possibilities being mutually exclusive, and, given the characterization of $\mathrm{Skel}(\TT_f)$ recalled at the beginning of this subsection, this implies directly that $C_W\cap \Gamma_W=\varnothing$. 
Moreover, if $a\neq b$ are elements of $\TT_\be$ such that $\bp_\be(a)=\bp_\be(b)=x$, then we can find $s\neq t\in[0,1)$ such that $p_\be(s)=a$, $p_\be(t)=b$. Then it holds that $d_W(s,t)=0$, but $d_\be(s,t)>0$, which implies that $d_Z(s,t)=0$, meaning that $x\in \Gamma_W$. The argument is completely symmetric and can be applied to give the last property.  
\end{proof}

\begin{figure}
    \centering
    \includegraphics[width=0.5\linewidth]{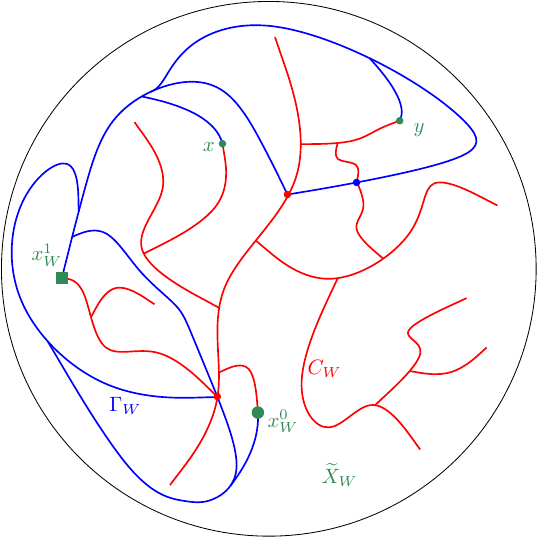}
    \caption{Some branches of the tree $\Gamma_W$ of relative interiors of geodesics toward $x_W^1$ are represented in blue, and a part of the the cut locus $C_W$ is represented in red. The two distinguished points $x_W^0$ and $x_W^1$ are almost surely not in $\Gamma_W\cup C_W$, and we have represented two more points $x,y$ outside this set.}
    \label{fig:twospheres}
\end{figure}

\subsection{Recovering the tree branches}
\label{sec:recov-tree-branches}

Working on the almost sure event that the properties of  \cref{lem:geominterpr} hold, one can define distinguished simple paths in the
Brownian sphere. 
For $x, y\in X_W\backslash \Gamma_W$, 
we let
\[
C_W(x,y)=  \bp_\be([[\bp_\be^{-1}(x),\bp_\be^{-1}(y)]]_\be). 
\]
Likewise, for $x,y\in
X_W\backslash C_W$, we define 
\[
\Gamma_W(x,y)=\bp_Z([[\bp^{-1}_Z(x),
\bp^{-1}_Z(y)]]_Z)\,.
\] 
Note that $C_W(x,y)$ and $\Gamma_W(x,y)$ are simple curves in $X_W$, by  \cref{lem:geominterpr}(i). 

\begin{lemma}\label{lem:branches}
Almost surely, it holds that:  
\begin{enumerate}[nosep,label=(\roman*)]
\item For every $x,y \in X_W\backslash C_W$, the two geodesics from
  $x$ and $y$ to $x_W^1$, meet for the first time at some point $z\in \Gamma_W$. 
  The set $\Gamma_W(x,y)$ is the union of the geodesic segment from $x$ to $z$ and the geodesic segment from $y$ to $z$, and
\begin{equation}\label{eq:distWh}
 d_Z(\bp_Z^{-1}(x),\bp_Z^{-1}(y))= d_W(x,z)+d_W(y,z).
\end{equation}
\item For every $x, y\in X_W\backslash \Gamma_W$, the set 
$$ C_W(x,y) \backslash\{x,y\}\subset C_W$$ 
is the set of points $z$ such that there exist two geodesics from $z$
to $x_W^1$ that separate $x$ and $y$ (meaning that the points $x$ and $y$ 
lie in different connected components of the complement of the geodesics).
\end{enumerate}
\end{lemma}

This lemma is illustrated in \cref{fig:twocurves}. 

\begin{figure}
    \centering
    \includegraphics[width=0.4\linewidth]{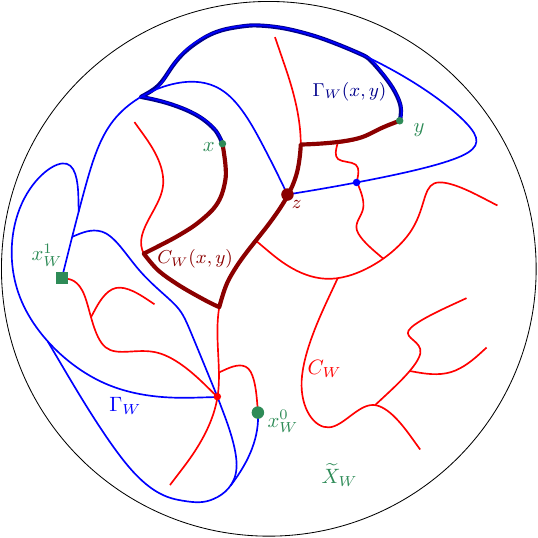}
    \caption{From the points $x,y\in \widetilde{X}_W$, there is a unique geodesic pointing towards $x^1_W$, which allows to identify $\Gamma_W(x,y)$, represented by the thick dark blue line. 
    The curve $C_W(x,y)$, represented by the thick dark red line, consists of points from which we can find (at least) two geodesics pointing towards $x^1_W$ that separate $x$ from $y$. One of these points, called $z$, is highlighted together with the two relevant geodesics.}
    \label{fig:twocurves}
\end{figure}

\begin{proof}
    The first point (i) is proved in \cite[Lemmas 4 and 5]{Mie13}, but we recall the main ideas. We use the fact that $\Gamma_W(x,x_W^1)$ is a geodesic path in $(X_W,d_W)$, isometric to the segment $[[\bp_Z^{-1}(x),\rho_Z]]_Z$ in $(\TT_Z,d_Z)$. The fact that $\TT_Z$ is an $\R$-tree shows that the paths $\Gamma_W(x,x_W^1)$ and $\Gamma_W(y,x_W^1)$ coincide on a maximal final segment of the form $\Gamma_W(z,x_W^1)$. Moreover, one has 
    \begin{align*}
    d_Z(\bp_Z^{-1}(x),\bp_Z^{-1}(y))&=d_Z(\bp_Z^{-1}(x),\bp_Z^{-1}(z))+d_Z(\bp_Z^{-1}(y),\bp_Z^{-1}(z))\\
    &=d_W(x,z)+d_W(y,z).
    \end{align*} 

    For the second point (ii), we first observe that if $z\notin C_W(x,y)\setminus\{x,y\}$, then the continuous path $C_W(x,y)$ connects $x$ and $y$ without intersecting $z$ (except possibly at its extremities) nor $\Gamma_W$ because of  \cref{lem:geominterpr}(iv). In particular, the geodesics from $z$ do not disconnect $x$ from $y$. 
    
    Conversely, let $z\in C_W(x,y)\setminus \{x,y\}$, and let $a=\bp_\be^{-1}(x)$ and $b=\bp_\be^{-1}(y)$. Note that $c=\bp_\be^{-1}(z)$, being an element of $\mathrm{Skel}(\TT_\be)$, splits the latter into two or three connected components. Let us assume for simplicity that we are in the first case (the case of three components can be treated similarly). This means that there exist exactly two times $u\neq u'\in (0,1)$ such that $p_\be(u)=p_\be(u')=c$. Our goal is to show that the union of the two geodesic segments $\bp_Z([[p_Z(u),\rho_Z]]_Z)$ and $\bp_Z([[p_Z(u'),\rho_Z]]_Z)$
    from $z$ to $x_W^1$ in $(X_W,d_W)$ separates $x$ from $y$. Note that these two paths form a ``lollipop'' shape: they both start from $z$, remain disjoint in a neighborhood of $z$, and then merge into a final common segment, and consequently, they indeed separate $X_W$ into two connected components. 
    
    In order to show that $x$ and $y$ do not belong to the same connected component, we consider a continuous path $\alpha:[0,1]\to X_W$ from $x$ to $y$, and aim to show that $\alpha$ intersects the union of geodesics from $z$ to $x_W^1$ in $(X_W,d_W)$. Let us call $\TT_\be(a)$ (resp.\ $\TT_\be(b)$) the component of $\TT_\be\setminus \{c\}$ containing $a$ (resp.\ $b$). Let $r_0=\inf\{r\in [0,1]:\bp_\be^{-1}(\alpha(r))\in \TT_\be(b)\}$. Then there exist sequences $(r_n),(r'_n)$ that are respectively non-decreasing and non-increasing with limit $r_0$ such that $\bp_\be^{-1}(\alpha(r_n))\in \TT_\be(a)$ and $\bp_\be^{-1}(\alpha(r_n'))\in \TT_\be(b)$ for every $n$, and up to extracting a subsequence, we may assume that the latter sequences converge to $a'\in \TT_\be(a)\cup \{c\}$ and $b'\in \TT_\be(b)\cup\{c\}$, such that $\bp_\be(a')=\bp_\be(b')=\alpha(r_0)$. 
    
    If $\alpha(r_0)=z$, then we obtain the desired conclusion. Otherwise, we have $a'\in \TT_\be(a)$ and $b'\in \TT_\be(b)$, and this implies the existence of $s',t'\in [0,1]$ such that $p_\be(s')=a'$ and $p_\be(t')=b'$, and $d_W(s',t')=0$ (since $a'$ and $b'$ both project to $\alpha(r_0)$ via $\bp_\be$). We know that this implies that $d_\be(s',t')=0$ or $d_Z(s',t')=0$, but the first case is excluded since $a'\neq b'$. This means that $Z_{s'}=Z_{t'}=\min_{[s',t']_\circ}Z$ or $Z_{s'}=Z_{t'}=\min_{[t',s']_\circ}Z$, depending on which of the minima is the largest one (recalling the notation for cyclic intervals from \cref{sec:trees-functions}). 
    Since $c$ is on the path of $\TT_\be$ from $a$ to $b$, it must hold that one of $u$ or $u'$ is in $[s',t']_\circ$, and  the other one is in $[t',s']_\circ$. This implies that $p_W(s')=p_W(t')$ is on the simple geodesic from $u$ or from $u'$, as desired. 
\end{proof}

\subsection{Orienting the sphere and recovering $Z$}
\label{sec:recovering-z}

In this section, we explain how one can recover the process $Z$ from $(\bfX^{2\bullet}_W,\epsilon_W)$. 
Recall that, a.s., $x_W^0,x_W^1\in \tX_W$. 
The
preceding lemma allows to describe $C_W$, $\Gamma_W$, $\tX_W$,
and the segments
$C_W(x,y)$ and $\Gamma_W(x,y)$ for $x,y\in \tX_W$, in terms of
$(X_W,d_W,\mu_W,x_W^0,x_W^1)$, 
in a way that is invariant with
respect to isometries. That is, if we select an element
$(X,d,\mu,x^0,x^1)$ in the isometry class
$\bfX_W^{2\bullet}$, then we can define sets
$C,\Gamma,\tX$, and $C(x,y)$ and $\Gamma(x,y)$ for every $x,y\in
\tX$, in such a way that a measure-preserving isometry  from
$(X_W,d_W,\mu_W,x_W^0,x_W^1)$ 
to $(X,d,\mu,x^0,x^1)$ sends $C_W$ to
$C$, $\Gamma_W$ to $\Gamma$, etc. This will be crucial in our main
proof.

\begin{lemma}
Almost surely, for every $x\in \tX_W\backslash\{x_W^0\}$, 
the curve $\gamma(x)=C_W(x_W^0,x)\cup \Gamma_W(x,x_W^0)$ is a closed, 
simple loop in $X_W$, and 
$p_W([0,p_W^{-1}(x)])$ and $p_W([p_W^{-1}(x),1])$ 
are the closures of the two Jordan domains that it separates.
\end{lemma}

\begin{proof}
Let $x\in \tX_W$ and $s=p_W^{-1}(x)$. Observe 
that the two closed simple
curves $C_W(x_W^0,x)$ and $\Gamma_W(x,x_W^0)$ intersect only at
their extremities, since 
$C_W\cap \Gamma_W=\varnothing$ and $x_W^0\in \tX_W$. 
As a consequence, their union $\gamma(x)$ forms a simple loop
separating two disks in $X_W$. Furthermore, the sets $p_W([0,s])$
and $p_W([s, 1])$ are compact and connected, with union $X_W$. Therefore, it
suffices to show that 
their boundary is $\gamma(x)$.

To this end, we proceed similarly to the proof of \cref{lem:branches}. We let $\alpha:[0,1]\to X_W$ be a continuous path with $\alpha(0)\in p_W([0,s])\setminus \gamma(x)$ and $\alpha(1)\in p_W([s,1])\setminus \gamma(x)$, and we aim to prove that $\alpha$ intersects $\gamma(x)$. Let $r_0=\inf\{r\in [0,1]:\alpha(r)\in p_W([s,1])\}$. Then, by compactness, there must exist $s'\in [0,s]$
and $t'\in [s,1]$ such that $p_W(s')=p_W(t')=\alpha(r_0)$. If $s'$ or $t'$ equals $s$, then this means that $\alpha(r_0)=x$ and we are done. Otherwise, we know that either $d_\be(s',t')=0$ or $d_Z(s',t')=0$. 
In the first case, we deduce $\alpha(r_0)\in C_W(x_W^0,x)$, 
and in the second
case,  $\alpha(r_0)\in \Gamma_W(x,x_W^0)$, which concludes the proof.
\end{proof}

\begin{figure}
    \centering
    \includegraphics[width=0.4\linewidth]{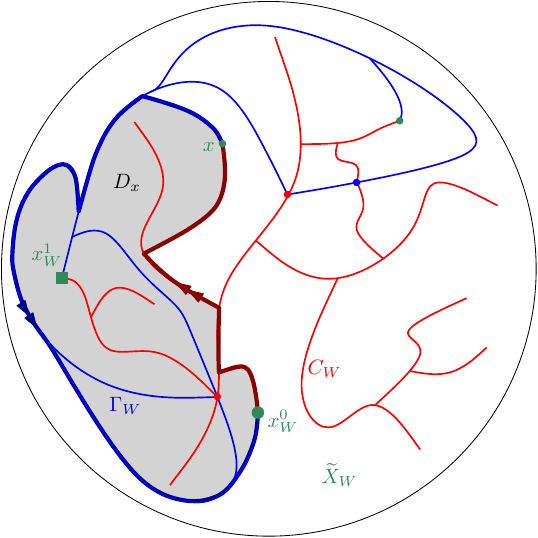}
    \caption{The curve $\gamma(x)$ is illustrated in thick lines, and the domain $D_x=p_W([0,p_W^{-1}(x)])$ is the gray area. If we choose to orient the curve $\gamma(x)$ by first following $C_W(x_W^0,x)$ from $x_W^0$ to $x$, then the Brownian sphere is canonically oriented in such a way that $D_x$ is circled counterclockwise by $\gamma(x)$, for any choice of $x\in \widetilde{X}_W\setminus \{x_W^0\}$. }
    \label{fig:twotrees}
\end{figure}

For a given $x\in \tX_W\backslash\{x_W^0\}$, we may choose
the orientation of $X_W$ such that the loop $\gamma(x)$, oriented by following the segment 
$C_W(x_W^0,x)$ from $x_W^0$ to $x$ and $\Gamma_W(x,x_W^0)$
from $x$ to $x_W^0$, goes around $p_W([0,p_W^{-1}(x)])$ counterclockwise.

We claim that this choice of orientation is independent of
$x$. To see this, note that if
$p_W^{-1}(x)<p_W^{-1}(x')$, we have that 
\[
p_W([0,p_W^{-1}(x)])\subset
p_W([0,p_W^{-1}(x')]), 
\]
and that
the paths $C_W(x_W^0, x)$ and $C_W(x_W^0,x')$
have nontrivial intersection, and similarly for $\Gamma_W(x,x_W^0)$ and
$\Gamma_W(x',x_W^0)$. The argument is similar if $p_W^{-1}(x')<p_W^{-1}(x)$. 
See \cref{fig:twotrees} for an illustration of the above discussion. 

This being said, we cannot a priori distinguish the two regions only from the isometry class $\bfX_W^{2\bullet}$. 
This is where the sign variable $\epsilon_W$ defined at \eqref{eq:epsilon-h} comes into play. In order
to make this specific choice of orientation, we take 
$x=x_W^1$ as a reference point, and decide that, 
among the two regions
separated by  $\gamma(x_W^1)$, the  region
circled counterclockwise by $\gamma(x_W^1)$ 
is the one of smallest $\mu_W$-measure
if $\epsilon_W=1$, or the one of largest $\mu_W$-measure if $\epsilon_W=-1$. There is an ambiguity when both regions have the same measure, but this happens with probability $0$ since $s_*=s_*(W)$ is a uniform random variable in $[0,1]$. In this way, we see that it is indeed possible to identify the region $p_W([0,s_*])$, and hence, the canonical orientation of $X_W$. 

Once this orientation choice is made, we can finally recover the process $Z$ as follows. 
For each $x\in \tX_W$, 
we define $D_x$ to be the disk in $X_W$ 
bounded by $C_W(x_W^0,x) \cup \Gamma_W(x,x_W^0)$, 
and such that the loop 
$\gamma(x)$
goes around $D_x$ in counterclockwise order, for the canonical orientation. From the above discussion, it holds that $D_x=p_W([0,p_W^{-1}(x)])$, and therefore, we see that we can identify $p_W^{-1}(x)$ purely in metric-measure terms, as follows. 
\begin{corollary} Almost surely, 
it holds that 
$$\forall x\in \tX_W,\qquad p_W^{-1}(x)= \mu_W(D_x).$$
\end{corollary}
We have thereby identified $p_W^{-1}(\tX_W)$, the restriction of $p_W$ to
$p_W^{-1}(\tX_W)$, and the restriction of $Z$ to
$p_W^{-1}(\tX_W)$, using \eqref{eq:distWh} and the observation that 
\[
Z_s-Z_{s_*(W)}=d_Z(\bp_Z^{-1}(p_W(s)),\bp_Z^{-1}(x_W^1))\, ,\quad s\in p_W^{-1}(\tX_W). 
\]
The functions $p_W$ and $Z$ can then be extended by continuity to $[0,1]$. 

We conclude this section by noting, as stressed in the previous
section, that our constructions are invariant with
respect to isometries. That is, applying them to any choice of 
element in the isometry class $\bfX_W^{2\bullet}$ produces the same
function $Z$.

\subsection{Recovering $\be$}

It remains to determine $\be$ on $[0,1]$, which requires a slightly different argument since there is no formula similar to~\eqref{eq:distWh} for expressing the distance $d_\be$ directly in terms of $[X_W, d_W]$.  
As noted, the idea here is that $\be$ encodes a known continuum tree, and we are also provided with the values of the Brownian snake parametrized by that tree. 
Thus in order to recover $\be$ we just need to know the length of paths in the tree.
These can be deduced from the Brownian snake by measuring its quadratic variation. 
See \cref{lem:recov-durat-proc} for this last step.
We proceed to make this precise.

First, we note that it suffices to determine $\be$ on 
$\mathbb{Q}\cap[0,1]$, as $\be$ can then be extended to $[0,1]$ by continuity. 
Observe that $p_W( \mathbb{Q}\cap [0,1])\subset \tX_W$, almost surely\footnote{There is a small subtlety here: the information given by the dense
  set $p_W(\mathbb{Q}\cap[0,1])$ is not obviously invariant under isometries. The reason why this indeed holds is that we 
  have identified the parametrization by $[0,1]$ (i.e., the measure 
  and order structure) in the previous subsection, 
  so we 
  can define ``the point visited at time $q\in \mathbb{Q}$'' by $p_W$ in purely measure-metric terms.}.
Consider the function $\ell$ on $X_W$ defined by 
\[
\ell(x)=d_W(x,x_W^1)-d_W(x_W^0, x_W^1).
\]
Fix $q\in \mathbb{Q}\cap [0,1]$, and let $\beta_q(r)=\sup\{t\leq q:\be_t=r\}$ for $0\leq r\leq \be_q$. Then $p_\be\circ \beta_q$ is isometric with image $[[\rho_\be,p_\be(q)]]_\be$ in $(\TT_\be,d_\be)$, so that $p_W\circ \beta_q$ is a natural parametrization of $C_W(x_W^0, p_W(q))$ by $[0,
\be_q]$. By the defining properties of 
$(\be,Z)$ and formula~\eqref{eq:distWh}, the mapping $B^{(q)}=\ell\circ p_W\circ \beta_q$ is a standard Brownian motion 
with duration $\be_q$, conditionally on $\be$.

Since we do not know $\be$ a priori from $(\bfX^{2\bullet}_W,\epsilon_W)$, we do not have direct access to $\beta_q$. However, if $\alpha_q$ is any parametrization of $C_W(x_W^0, p_W(q))$ (by $[0, 1]$, say), then $\ell\circ \alpha_q$
is a time-change of this Brownian motion. This means that $(\bfX^{2\bullet}_W,\epsilon_W)$ determines $B^{(q)}$ up to reparametrization, for every $q\in \mathbb{Q}\cap [0,1]$. 

\begin{lemma}\label{lem:recov-durat-proc}
Let $B=(B_t,t\geq 0)$ be a standard Brownian motion.
Then, almost surely, the following property holds. For every increasing
continuous function $\kappa:[0,1]\to \mathbb{R}_+$ with $\kappa(0)=0$, it is
possible to compute $\kappa(1)$ as a function of $(B_{\kappa(u)},0\leq
u\leq 1)$. 
\end{lemma}

\begin{proof}
Consider the number of steps made by $(B_{\kappa(u)},0\leq u\leq 1)$
on the sub-lattice $\eps\Z$, multiply by $\eps^2$, 
and then take a limit. 
\end{proof}

From this observation, we deduce that almost surely, $(\be_q,q\in \mathbb{Q}\cap[0,1])$ can be recovered from the isometry class $\bfX_W^{2\bullet}$ and $\epsilon_W$.  
Finally, we conclude that $W=(\be,Z)$ is
a function of $\epsilon_W$ and the isometry class $\bfX_W^{2\bullet}$, rather than of
some particular representative. This completes the proof.

\makeatletter
\renewcommand\@biblabel[1]{#1.}
\makeatother


\begin{thebibliography}{10}

\bibitem{Ald91_I}
D.~Aldous.
\newblock The continuum random tree. {I}.
\newblock {\em Ann. Probab.}, 19(1):1--28, 1991.

\bibitem{Ald91_II}
D.~Aldous.
\newblock The continuum random tree. {II}. {A}n overview.
\newblock In {\em Stochastic analysis ({D}urham, 1990)}, volume 167 of {\em
  London Math. Soc. Lecture Note Ser.}, pages 23--70. Cambridge Univ. Press,
  Cambridge, 1991.

\bibitem{Ald93_III}
D.~Aldous.
\newblock The continuum random tree. {III}.
\newblock {\em Ann. Probab.}, 21(1):248--289, 1993.

\bibitem{BaMiRa19}
E.~Baur, G.~Miermont, and G.~Ray.
\newblock Classification of scaling limits of uniform quadrangulations with a
  boundary.
\newblock {\em Ann. Probab.}, 47(6):3397--3477, 2019.

\bibitem{BeMi22}
J.~Bettinelli and G.~Miermont.
\newblock Compact {B}rownian surfaces {II}. {O}rientable surfaces.
\newblock Available at
  \href{https://arxiv.org/abs/2212.12511}{arXiv:2212.12511}.

\bibitem{BeMi17}
J.~Bettinelli and G.~Miermont.
\newblock Compact {B}rownian surfaces {I}: {B}rownian disks.
\newblock {\em Probab. Theory Related Fields}, 167(3-4):555--614, 2017.

\bibitem{ChUB15}
L.~Chaumont and G.~Uribe~Bravo.
\newblock Shifting processes with cyclically exchangeable increments at random.
\newblock In {\em X{I} {S}ymposium on {P}robability and {S}tochastic
  {P}rocesses}, volume~69 of {\em Progr. Probab.}, pages 101--117.
  Birkh\"auser/Springer, Cham, 2015.

\bibitem{CV81}
R.~Cori and B.~Vauquelin.
\newblock Planar maps are well labeled trees.
\newblock {\em Canadian J. Math.}, 33(5):1023--1042, 1981.

\bibitem{CLG14}
N.~Curien and J.-F. Le~Gall.
\newblock The {B}rownian plane.
\newblock {\em J. Theoret. Probab.}, 27(4):1249--1291, 2014.

\bibitem{kechris95}
A.~S. Kechris.
\newblock {\em Classical descriptive set theory}, volume 156 of {\em Graduate
  Texts in Mathematics}.
\newblock Springer-Verlag, New York, 1995.

\bibitem{LG93}
J.-F. Le~Gall.
\newblock A class of path-valued {M}arkov processes and its applications to
  superprocesses.
\newblock {\em Probab. Theory Related Fields}, 95(1):25--46, 1993.

\bibitem{legall07}
J.-F. Le~Gall.
\newblock The topological structure of scaling limits of large planar maps.
\newblock {\em Invent. Math.}, 169(3):621--670, 2007.

\bibitem{LG10}
J.-F. Le~Gall.
\newblock Geodesics in large planar maps and in the {B}rownian map.
\newblock {\em Acta Math.}, 205(2):287--360, 2010.

\bibitem{LG13}
J.-F. Le~Gall.
\newblock Uniqueness and universality of the {B}rownian map.
\newblock {\em Ann. Probab.}, 41(4):2880--2960, 2013.

\bibitem{LG22}
J.-F. Le~Gall.
\newblock The volume measure of the {B}rownian sphere is a {H}ausdorff measure.
\newblock {\em Electron. J. Probab.}, 27:Paper No. 113, 28, 2022.

\bibitem{LGP08}
J.-F. Le~Gall and F.~Paulin.
\newblock Scaling limits of bipartite planar maps are homeomorphic to the
  2-sphere.
\newblock {\em Geom. Funct. Anal.}, 18(3):893--918, 2008.

\bibitem{LGW06}
J.-F. Le~Gall and M.~Weill.
\newblock Conditioned {B}rownian trees.
\newblock {\em Ann. Inst. H. Poincar\'{e} Probab. Statist.}, 42(4):455--489,
  2006.

\bibitem{miermont09}
G.~Miermont.
\newblock Tessellations of random maps of arbitrary genus.
\newblock {\em Ann. Sci. \'Ec. Norm. Sup\'er. (4)}, 42(5):725--781, 2009.

\bibitem{Mie13}
G.~Miermont.
\newblock The {B}rownian map is the scaling limit of uniform random plane
  quadrangulations.
\newblock {\em Acta Math.}, 210(2):319--401, 2013.

\bibitem{miermont14}
G.~Miermont.
\newblock Aspects of random maps.
\newblock Lecture notes of the Saint-Flour Summer School 2014, available at
  \href{http://perso.ens-lyon.fr/gregory.miermont/coursSaint-Flour.pdf}{http://perso.ens-lyon.fr/gregory.miermont/coursSaint-Flour.pdf},
  2014.

\bibitem{MS20_I}
J.~Miller and S.~Sheffield.
\newblock Liouville quantum gravity and the {B}rownian map {I}: the {${\rm
  QLE}(8/3,0)$} metric.
\newblock {\em Invent. Math.}, 219(1):75--152, 2020.

\bibitem{MS21_II}
J.~Miller and S.~Sheffield.
\newblock Liouville quantum gravity and the {B}rownian map {II}: {G}eodesics
  and continuity of the embedding.
\newblock {\em Ann. Probab.}, 49(6):2732--2829, 2021.

\bibitem{MS21_III}
J.~Miller and S.~Sheffield.
\newblock Liouville quantum gravity and the {B}rownian map {III}: the conformal
  structure is determined.
\newblock {\em Probab. Theory Related Fields}, 179(3-4):1183--1211, 2021.

\bibitem{Sch98}
G.~Schaeffer.
\newblock {\em Conjugaison d'arbres et cartes combinatoires al{\'e}atoires}.
\newblock PhD thesis, Universit{\'e} Bordeaux I, 1998.

\end{thebibliography}
\end{document}